\documentclass{amsart}

\usepackage{hyperref}
\usepackage{tikz}
\usepackage{lipsum}

\newtheorem{theorem}{Theorem}[section]
\newtheorem{lemma}[theorem]{Lemma}

\theoremstyle{definition}
\newtheorem{definition}[theorem]{Definition}

\theoremstyle{remark}

\DeclareMathOperator{\Gl}{GL}
\DeclareMathOperator{\Gal}{Gal}
\DeclareMathOperator{\Aut}{Aut}

\numberwithin{equation}{section}

\begin{document}

\title{Torsion growth of rational elliptic curves in sextic number fields}

\author{Tomislav Gu\v{z}vi\'c}
\thanks{The author gratefully acknowledges support
from the QuantiXLie Center of Excellence, a project co-financed by the Croatian Government and European Union through the
European Regional Development Fund - the Competitiveness and Cohesion Operational Programme (Grant KK.01.1.1.01.0004) and
by the Croatian Science Foundation under the project no. IP-2018-01-1313.}

\begin{abstract}
We classify the possible torsion structures of rational elliptic
curves over sextic number fields.
\end{abstract}

\maketitle

\section{Introduction}
Let $E/K$ be an elliptic curve defined over a number field $K$. The Mordell-Weil Theorem states that the set of $K$-rational points $E(K)$, is a finitely generated abelian group.
Denote by $E(K)_{tors}$ the torsion subgroup of $E(K)$, which is isomorphic to $C_m \oplus C_n$ for
two positive integers $m, n$, where $m$ divides $n$ and where $C_n$ is a cyclic group of order $n$.
\\
For a positive integer $d$,
\begin{enumerate}
    \item Let $\Phi(d)$ be the set of possible isomorphism classes of groups $E(K)_{tors}$, where $K$ runs through all number fields $K$ of degree $d$ and $E$ runs through all elliptic curves over $K$.
    \item Let $\Phi_{\mathbb{Q}}(d) \subseteq \Phi(d)$ be the set of possible isomorphism classes of groups $E(K)_{tors}$, where $K$ runs through all number fields $K$ of degree $d$ and $E$ runs through all elliptic curves defined over $\mathbb{Q}$.
    \item Let $\Phi^{\infty}(d)$ be the subset of isomorphism classes of groups $\Phi(d)$ that occur infinitely often. More precisely, a torsion group $G$ belongs to $\Phi^{\infty}(d)$ if there are infinitely many elliptic curves $E$, non-isomorphic over $\bar{\mathbb{Q}}$, such that $E(K)_{tors} \cong G.$
    \item Let $R_{\mathbb{Q}}(d)$ be the set of all primes $p$ such that there exists a number field $K$ of degree $d$, an elliptic curve $E/\mathbb{Q}$ such that there exists a point of order $p$ on $E(K)$.
    \item Let $\Phi^{\mathrm{CM}}(d)$ be the set of possible isomorphism classes of groups $E(K)_{tors}$, where $K$ runs through all number fields $K$ of degree $d$ and $E$ runs through all elliptic curves with $\mathrm{CM}$ over $K$.
\end{enumerate}   

The set $\Phi(1)$ was determined by Mazur in \cite{11}. Kenku, Momose and Kamienny determined $\Phi(2)$ in \cite{20}, \cite{13}. For $d \ge 3$, the set $\Phi(3)$ has been determined by Derickx, Etropolski, Hoeij, Morrow and Zureick-Brown but the result is (at the time of writing this paper) unpublished. In \cite{21}, Merel proved that $\Phi(d)$ is finite for all positive integers $d$. It is known that $\Phi(1)=\Phi^{\infty}(1)$ and $\Phi(2)=\Phi^{\infty}(2)$.
\\
The set $\Phi^{\infty}(d)$ has been determined for $d=3$ by Jeon, Kim and Schweizer \cite{25}, for $d=4$ by Jeon, Kim and Park \cite{26} and for $d=5,6$ by Derickx and Sutherland \cite{27}. 
\\
In \cite{7}, Najman has determined the sets $\Phi_{\mathbb{Q}}(2)$ and $\Phi_{\mathbb{Q}}(3)$.$\Phi_{\mathbb{Q}}(4)$ has been determined by Chou \cite{22} and Najman and Gonz\'alez-Jim\'enez \cite{2}. Najman and Gonz\'alez-Jim\'enez also determined $\Phi_{\mathbb{Q}}(p)$, for $p \ge 7$ prime number in \cite{2}. Gonz\'alez-Jim\'enez has determined $\Phi_{\mathbb{Q}}(5)$ in \cite{24}. The set $R_{\mathbb{Q}}(d)$, for $d \le 3342296$ has been determined by Najman and Gonz\'alez-Jim\'enez in \cite{2}. $\Phi^{\mathrm{CM}}(1)$ has been determined by Olson in \cite{32} and $\Phi^{\mathrm{CM}}(d)$ for $d=3,4$ by Zimmer and his collaborators.
The sets $\Phi^{\mathrm{CM}}(d)$, for $4\le d \le 13$ have been determined by Clark, Corn, Rice and Stankiewicz in \cite{30}.

The main result of this paper is the following theorem.
\begin{theorem}
	\label{Theorem 1.1}
Let $E/\mathbb{Q}$ be an elliptic curve and let $K$ be a sextic number field. Then
\[ E(K)_{tors} \cong
\begin{cases} 
      C_m, & m=1,...,16, 18, 21, 30, m \ne 11,\\
      C_{2} \oplus C_{2m}, & m=1,...,7, 9,\\
      C_3 \oplus C_{3m}, & m=1,...,4, \\
      C_4 \oplus C_{4m}, & m=1,3, \\
      C_6 \oplus C_6, \\
      C_3 \oplus C_{18}.
\end{cases} \]
\end{theorem}

It is known that all groups mentioned in Theorem 1.1 except $C_3 \oplus C_{18}$ appear for some elliptic curve $E/\mathbb{Q}$ and some sextic field $K$. We expect that the group $C_3 \oplus C_{18}$ never occurs in this situation but we're unable to completely prove it. We give partial result regarding this group in the last theorem of this paper.
\\
\\
In order to prove Theorem 1.1, we heavily rely on results about possible images of Galois representations attached to $E$ by Zywina in \cite{3} (see also \cite{19})  and on possible values of $[\mathbb{Q}(P):\mathbb{Q}]$ (classified by Najman and González-Jim\'enez in \cite{2}), where $P \in E(\overline{\mathbb{Q}})$ is a point of prime order $p$. We will often have information on mod $p$ and mod $q$ Galois representations attached to $E$, where $p$ and $q$ are different prime numbers and we will clasify rational points on the corresponding modular curve.
\\
\\
In our computation, we used \texttt{Magma} \cite{35}. \texttt{Magma} \cite{35}  code used in this paper is available on this paper's $\mathrm{arXiv}$ webpage. Many functions we used were taken from Enrique Gonz\'alez-Jim\'enez's webpage. 
\section{Notation and auxiliary results}

\begin{theorem}[Mazur, \cite{11}]
\label{Theorem 2.1}
Let $E/\mathbb{Q}$ be an elliptic curve. Then

\[ E(\mathbb{Q})_{tors} \cong
\begin{cases} 
      C_m, & m=1,...,10, 12, \\
      C_{2} \oplus C_{2m}, & m=1,...,4.
   \end{cases} \]
\end{theorem}

\begin{theorem}[Kenku, Momose, \cite{20}, Kamienny \cite{13}]
\label{Theorem 2.2}
Let $E/F$ be an elliptic curve over a quadratic number field $F$. Then 
\[ E(F)_{tors} \cong
\begin{cases} 
      C_m, & m=1,...,16,18, \\
      C_{2} \oplus C_{2m}, & m=1,...,6,\\
      C_3 \oplus C_{3m}, & m=1,2, \\
      C_4 \oplus C_4.
   \end{cases} \]
   
\end{theorem}
Let $E/F$ be an elliptic curve defined over a number field $F$. There exists an $F$-rational cyclic isogeny $\phi : E \to E'$ of degree $n$ if and only if $\ker\phi$ is a $Gal(\Bar{F}/F)$-
invariant cyclic group of order $n$; in this case we say that $E$ has an $F$-rational $n$-isogeny.
When $F=\mathbb{Q}$, possible degrees $n$ of elliptic curves over $\mathbb{Q}$ are known by the following theorem.
\begin{theorem}[Mazur \cite{12}, Kenku \cite{14}, \cite{15}, \cite{16}, \cite{17}]
\label{Theorem 2.3}
Let $E/\mathbb{Q}$ be an elliptic curve with a rational $n$-isogeny. Then \[ n \in \{ 1,...,19,21, 25,27, 37, 43, 67, 163 \} .\]
There are infinitely many elliptic curves (up to $\bar{\mathbb{Q}}$-isomorphism) with a rational $n$-isogeny over $\mathbb{Q}$ for \[ n \in \{
{1, . . . , 10, 12, 13, 16, 18, 25} \} \] and only finitely many for all the other $n$.
\end{theorem}

\begin{theorem} [Najman, \cite{7}, Theorem 2]
\label{Theorem 2.4}
Let $E/\mathbb{Q}$ be an elliptic curve and $F$ a quadratic field. Then 
\[ E(F)_{tors} \cong
\begin{cases} 
      C_m, & m=1,...,10,12,15,16, \\
      C_{2} \oplus C_{2m}, & m=1,...,6,\\
      C_3 \oplus C_{3m}, & m=1,2, \\
      C_4 \oplus C_4.
   \end{cases} \]
\end{theorem}

\begin{theorem} [Najman, \cite{7}, Theorem 1]
\label{Theorem 2.5}
Let $E/\mathbb{Q}$ be an elliptic curve and $K$ a cubic field. Then 
\[ E(K)_{tors} \cong
\begin{cases} 
      C_m, & m=1,...,10, 12, 13, 14, 18, 21, \\
      C_{2} \oplus C_{2m}, & m=1,...,4,7.
   \end{cases} \]

\end{theorem}

\begin{theorem}[\cite{30}]
Let $E/K$ be an elliptic curve with $\mathrm{CM}$ and let $K$ be a sextic number field. Then 
\[ E(K)_{tors} \cong
\begin{cases} 
      C_m, & m=1, 2, 3, 4, 6, 7, 9, 10, 14, 18, 19, 26, \\
      C_{2} \oplus C_{2m}, & m=2, 4, 6, 7, \\
      C_{3} \oplus C_{3m}, & m=1, 2, 3, \\
      C_6 \oplus C_6.
   \end{cases} \]
\end{theorem}

In \cite[Table 1]{1} it was shown that $\Phi_{\mathbb{Q}}(6)$ contains the following groups:
\begin{align*}
    C_{30}, \hspace{0.5cm}C_{2} \oplus C_{18}, \hspace{0.5cm}C_{3} \oplus C_{9}, \hspace{0.5cm}C_{3} \oplus C_{12}, \hspace{0.5cm}C_{4} \oplus C_{12}, \hspace{0.5cm}C_{6} \oplus C_{6}.
\end{align*}
Since $\Phi_{\mathbb{Q}}(6) \supseteq \Phi_{\mathbb{Q}}(2), \Phi_{\mathbb{Q}}(3)$, we have
\begin{align*}
    \Phi_{\mathbb{Q}}(6) \supseteq \Phi_{\mathbb{Q}}(2) \cup \Phi_{\mathbb{Q}}(3) \cup \{ C_{30}, C_{2} \oplus C_{18}, C_{3} \oplus C_{9}, C_{3} \oplus C_{12}, C_{4} \oplus C_{12}, C_{6} \oplus C_{6}. \}
\end{align*}

\begin{lemma}[\cite{2}, Section 5]
\label{Lemma 2.7}
The set $R_{\mathbb{Q}}(6)$ equals $\{ 2, 3, 5, 7, 13 \}.$
\end{lemma}
By the previous lemma, we only need to consider groups $C_m \oplus C_{n}$ such that the prime factors of $n$ are in $R_{\mathbb{Q}}(6)$.
\subsection*{Galois representations attached to elliptic curves}
Let $E/\mathbb{Q}$ be an elliptic curve and $n$ a positive integer. We denote by $E[n]$ the $n$-torsion subgroup of $E(\overline{\mathbb{Q}})$. The field $\mathbb{Q}(E[n])$ is the number field obtained by adjoining to $\mathbb{Q}$ all the $x$ and $y$-coordinates of the points of $E[n]$. The absolute Galois group $Gal(\overline{\mathbb{Q}}/\mathbb{Q})$ acts on $E[n]$ by its action on the coordinates of the
points, inducing a mod $n$ Galois representation attached to $E$:
\[ \rho_{E,n}: \Gal(\overline{\mathbb{Q}}/\mathbb{Q}) \to \Aut(E[n]) .\]
After fixing a base for the $n$-torsion, we identify $\Aut(E[n])$ with $\Gl_{2}(\mathbb{Z}/n\mathbb{Z})$.
This means that we can consider $\rho_{E,n}(\Gal(\overline{\mathbb{Q}}/\mathbb{Q}))$ as a subgroup of $\Gl_{2}(\mathbb{Z}/n\mathbb{Z})$, uniquely determined up to conjugacy. We shall denote $\rho_{E,n}(\Gal(\overline{\mathbb{Q}}/\mathbb{Q}))$ by $G_{E}(n)$. Moreover, since $\mathbb{Q}(E[n])$ is Galois extension of $\mathbb{Q}$ and $\ker \rho_{E,n}= \Gal(\overline{\mathbb{Q}}/\mathbb{Q}(E[n]))$, by the first isomorphism theorem we have $G_{E}(n) \cong \Gal(\mathbb{Q}(E[n])/\mathbb{Q})$.
\\
Rouse and Zureick-Brown \cite{31} have classified all the possible $2$-adic images of $\rho_{E,2^{\infty}}: \Gal(\overline{\mathbb{Q}}/\mathbb{Q}) \to \Gl_{2}(\mathbb{Z}_2)$, and have given explicitly all the $1208$ possibilities. We will use the same notation as in \cite{31} for the $2$-adic image of a given elliptic curve $E/\mathbb{Q}$. In \cite{5}, Gonz\'alez-Jim\'enez and Lozano-Robledo have determined for each possible image the degree of the field of definition of any $2$-subgroup.
From the results of \cite{36} one can see if a given $2$-subgroup is defined over a number field of given degree $d$. 

\subsection*{Division polynomial method}
 $E/\mathbb{Q}$ be an elliptic curve and $n$ a positive integer. We denote by $\psi_{E,n}$ the $n$-th division polynomial of $E$ (see \cite[Section 3.2]{33}). If $n$ is odd, then the roots of $\psi_{E, n}$ are precisely the $x$-coordinates of points $P \in E[n]$. Similarly, if $n$ is even, then the roots of $\psi_{E, n}/\psi_{E,2}$ are precisely the $x$-coordinates of points $P \in E[n]\setminus E[2]$.  $f_{E,n}$ denote the corresponding primitive division polynomial associated to $E$, i.e. it's roots are the $x$-coordinates of points $P$ on $E(\overline{\mathbb{Q}})$ of exact order $n$. We briefly describe the construction of $f_{E,n}$. If $n=p$ is prime, then $f_{E,p}=\psi_{E,p}$. For an arbitrary $n$, we have \[  f_{E,n}=\frac{\psi_{E,n}}{\displaystyle \prod_{d|n, d \ne n}f_{E,d}}.\] Note that if $E^{d}/\mathbb{Q}$ is a quadratic twist of $E/\mathbb{Q}$, then $\psi_{E,n}=\alpha \psi_{E^{d},n}$ and $f_{E,n}=\beta f_{E^d,n}$, for some rational constants $\alpha$, $\beta$. Consider the following problem:
\\
\\
\textit{Given a rational number $j$ and K a number field of degree $d$, does there exist an elliptic curve $E/\mathbb{Q}$ such that $j=j(E)$ can $E(K)$ contains a point $P$ of exact order $n$?}
\\
\\
 $E_{0}/\mathbb{Q}$ be any elliptic curve with $j=j(E_{0})$. In \texttt{Magma} \cite{35}, we compute primitive division polynomial $f_{E_0,n}$. Since every elliptic curve $E/\mathbb{Q}$ with $j(E)=j$ is a quadratic twist of $E_0$, we have $f_{E_0,n}=\beta f_{E,n}$, for some rational number $\beta$. Next, we factor $f_{E_0,n}$ over $\mathbb{Q}[x]$.  $d'$ denote the degree of the smallest irreducible factor $f$ of $f_{E_0,n}$ and  $x_0$ be a root of $f$. If $d' >d$, then $[\mathbb{Q}(P):\mathbb{Q}] \ge [\mathbb{Q}(x_0):\mathbb{Q}]=d' > d=[K:\mathbb{Q}]$ and so a point $P$ of exact order $n$ on $E(\overline{\mathbb{Q}})$ can't be defined over $K$.
\\
\\
Notation. Specific elliptic curves mentioned in this paper will be referred to by their $\mathrm{LMFDB}$ label and a link to
the corresponding $\mathrm{LMFDB}$ page \cite{18} will be included for the ease of the reader. Conjugacy classes of subgroups
of $\Gl_{2}(\mathbb{Z}/p\mathbb{Z})$ will be referred to by the labels introduced by Sutherland in \cite{19}.

\section{Classification of $\Phi_{\mathbb{Q}}(6)$}
In this section, we prove Theorem \ref{Theorem 1.1}.
From now on,  $K$ denote degree $6$ extension of $\mathbb{Q}$.
\begin{theorem}
 $E/\mathbb{Q}$ be an elliptic curve with $\mathrm{CM}$. Then $E(K)_{tors}$ is one of the groups listed in Theorem 1.1.
\end{theorem}
\begin{proof}
By Theorem 2.6, we see that the only groups contained in $\Phi^{\mathrm{CM}}(6)$ that do not appear in \begin{align*}
    \Phi_{\mathbb{Q}}(2) \cup \Phi_{\mathbb{Q}}(3) \cup \{ C_{30}, C_{2} \oplus C_{18}, C_{3} \oplus C_{9}, C_{3} \oplus C_{12}, C_{4} \oplus C_{12}, C_{6} \oplus C_{6} \}
\end{align*} are $C_{19}$ and $C_{26}$. By \cite[Proposition 7.(A)]{8}, both of these groups can't occur.
\end{proof}

\begin{lemma}\cite[Lemma 2.6, Lemma 2.8, Lemma 2.9]{1}
\label{Lemma 3.2}
 $E/\mathbb{Q}$ be an elliptic curve without CM. Then the following claims hold:
\begin{itemize}
    \item There are no points of order $l^2$, where $l \ge 11$ on an elliptic curve $E/\mathbb{Q}$ over any number field of degree $d < 55$.
    \item There are no points of order $49$ on an elliptic curve $E/\mathbb{Q}$ over any number field of degree $d < 42$.
    \item There are no points of order $81$ on an elliptic curve $E/\mathbb{Q}$ over any number field of degree $d < 81$.
\end{itemize}
\end{lemma}

\begin{lemma}\cite[Lemma 5]{8}
\label{Lemma 3.3}
 $E/\mathbb{Q}$ be an elliptic curve without CM, $K/\mathbb{Q}$ a sextic field and $P_{p} \in  E(K)_{tors}$ a
point of odd prime order $p$. Then $E$ has a rational $p$-isogeny, except if $E$ has $\mathrm{LMFDB}$ label $\href{http://www.lmfdb.org/EllipticCurve/Q/2450ba1/}{\mathrm{2450.y1}}$ or $\href{http://www.lmfdb.org/EllipticCurve/Q/2450bd1/}{\mathrm{2450.z1}}$,
and $p = 7$, where there are not rational $7$-isogenies. Moreover, in those last cases, the unique sextic fields where the torsion grows are $K = \mathbb{Q}(E[2])$ and $K' = \mathbb{Q}(P_7)$ ($K'/\mathbb{Q}$ is a non-Galois), where $E(K)_{tors} \cong C_2 \oplus C_2$ and $E(K)_{tors} \cong C_7$ respectively.
\end{lemma}

\begin{lemma}
\label{Lemma 3.4}
 $[K:\mathbb{Q}]=6$. If $L,L'$ are cubic subextensions of $K$ and if $L/\mathbb{Q}$ is Galois, then $L=L'$.
\end{lemma}
\begin{proof}
Assume that $L \neq L'$. Obviously $L \cap L'= \mathbb{Q}$ and $LL'=K$. By Galois theory we have $\Gal(K/L') \cong \Gal(LL'/L') \cong \Gal(L/\mathbb{Q})$. Since $|\Gal(K/L')|$ is $2$ and $|\Gal(L/\mathbb{Q})|=3$, we arrive at the contradiction.
\end{proof}

\begin{theorem}
 $E/\mathbb{Q}$ be an elliptic curve without CM. Then $E(K)_{tors}$ can't contain $C_{169}$, $C_{49}$, $C_{39}$, $C_{65}$, $C_{91}$, $C_{35}$, $C_{28}$, $C_{20}$, $C_{26}$, $C_{63}$, $C_{42}$, $C_{25}$, $C_{45}$, $C_{27}$, $C_{24}$.
\end{theorem}

\begin{proof}
If $E$ has $\mathrm{LMFDB}$ label $\href{http://www.lmfdb.org/EllipticCurve/Q/2450ba1/}{\mathrm{2450.y1}}$ or $\href{http://www.lmfdb.org/EllipticCurve/Q/2450bd1/}{\mathrm{2450.z1}}$, this holds by Lemma \ref{Lemma 3.3}. Suppose this isnt the case. By \ref{Lemma 3.2}, $E(K)$ can't contain $C_{169}$ and $C_{49}$. By Lemma \ref{Lemma 3.3}, if $E(K)$ contains points $P_p, P_q$ of odd prime orders $p$ and $q$, $p \neq q$, then $E(K)$ has a rational $p$ and $q$-isogenies, so it has a rational $pq$-isogeny. When $pq \in \{ 39, 65, 91, 35 \}$, this cannot happen, because of Theorem 2.3 and so $E(K)$ can't contain $C_{39}, C_{65}, C_{91}$ or $C_{35}$. In \cite[Proposition 6.s), k)]{8}, it has been proven that $E(K)$ can't contain $C_{28}$, $C_{20}$ or $C_{26}$, respectively. 
\\
    \fbox{$C_{63}$, $C_{42}$}: From Lemma \ref{Lemma 3.3} we conclude that $E$ has rational $3$ and $7$ isogenies, so it has a rational $21$-isogeny, so $j(E) \in \{-3^2\cdot5^6/2^3, 3^3\cdot5^3/2, 3^3\cdot5^3\cdot101^3/2^{21}, -3^3\cdot5^3\cdot383^3/2^7 \}$. For each of the possible $j$-invariants, using division polynomial method in \texttt{Magma} \cite{35}, we define polynomials $f_{63}$ and $f_{42}$ whose roots are $x$-coordinates of points on $E$ of order exactly $63$ and $42$, respectively. None of these polynomials has an irreducible factor of degree less then or equal to $6$. Hence, a point $P$ of order $63$ (resp. $42$) can't be defined over $K$.
    \\
    \\
    \fbox{$C_{25}$}: By Lemma \ref{Lemma 3.3}, $E$ has a rational $5$-isogeny. By \cite[Table 2]{2}, we see that $G_{E}(5) \in \{ \mathrm{5Cs.1.1, 5Cs.1.3, 5Cs.4.1, 5B.1.1, 5B.1.4, 5B.4.1} \}$. For each of these posibilities of $G_{E}(5)$, we find all subgroups $G$ of $\Gl_{2}(\mathbb{Z}/25\mathbb{Z})$ with surjective determinant that reduce to $G_{E}(5)$ modulo $5$. Then for each vector $v \in (\mathbb{Z}/25\mathbb{Z})^2$ of order $25$ we calculate the index of $G_v$ in $G$, where $G_v$ is stabiliser subgroup corresponding to vector $v$. By Theorem \ref{Theorem 2.1}, Theorem \ref{Theorem 2.4} and Theorem \ref{Theorem 2.5} we have that $[\mathbb{Q}(P_{25}):\mathbb{Q}] \notin \{ 1, 2, 3 \}$, so we have $[\mathbb{Q}(P_{25}):\mathbb{Q}]=6$. This means that $[G:G_v]=6$. Computation in \texttt{Magma} \cite{35} shows that this does not occur. Therefore, $E$ cannot have a point $P_{25}$ defined over $K$.
    \\
    \\
    \fbox{$C_{45}$}: Since $E$ has rational $3$ and $5$ isogenies by lemma 3.2, $E$ has a rational $15$-isogeny, so $j(E) \in \{ -5^2/2, -5^2 \cdot241^3/2^3, -29^{3} \cdot 5/2^{5}, 211^3 \cdot 5/2^{15} \}$. Using exactly the same method as in the $C_{63}$ and $C_{42}$ case, we find that a point of order $45$ can't be defined over $K$.
    \\
    \\
    \fbox{$C_{27}$}:  $P_{27}$ be a point of order $27$ in $E(K)$ and $P_{81}$ be a point of order $81$ in $E(\Bar{\mathbb{Q}})$ such that $3P_{81}=P_{27}$. From \cite[Proposition 4.6.]{2}, we have $[\mathbb{Q}(P_{81}):\mathbb{Q}(P_{27})] \le 9$. Since $[\mathbb{Q}(P_{27}):\mathbb{Q}] \le 6$, we have $[\mathbb{Q}(P_{81}): \mathbb{Q}]=[\mathbb{Q}(P_{81}):\mathbb{Q}(P_{27})] \cdot [\mathbb{Q}(P_{27}):\mathbb{Q}] \le 54$. This contradicts Lemma \ref{Lemma 3.2}.
    \\
    \\
    \fbox{$C_{24}$}: By Lemma \ref{Lemma 3.3}, $\rho_{E,3}$ is not surjective. $\rho_{E,8}$ can't be surjective because a point $P_8$ of order $8$ on $E(\overline{\mathbb{Q}})$ would satisfy $[\mathbb{Q}(P_8):\mathbb{Q}] > 6$.  By \cite[Theorem A (3)]{4}, we have that $G_{E}(8) \subseteq H$, for $H \in \{ \href{http://users.wfu.edu/rouseja/2adic/X30.html}{H_{30}}, \href{http://users.wfu.edu/rouseja/2adic/X31.html}{H_{31}}, \href{http://users.wfu.edu/rouseja/2adic/X39.html}{H_{39}}, \href{http://users.wfu.edu/rouseja/2adic/X45.html}{H_{45}}, \href{http://users.wfu.edu/rouseja/2adic/X47.html}{H_{47}}, \href{http://users.wfu.edu/rouseja/2adic/X50.html}{H_{50}} \}$. Each of these six groups has an order equal to $128$. This means that $[\mathbb{Q}(E[8]):\mathbb{Q}]$ is a power of $2$ that divides $128$. Consequently, $[\mathbb{Q}(E[2]):\mathbb{Q}]$ is a power of $2$. Hence, each $2$-torsion point on $E$ is defined over at most quadratic extension of $\mathbb{Q}$. Since $2^k$-torsion grows in extensions of degree $1,2$ or $4$ (\cite[Proposition 4.8)]{2})  and since $E(K) \supseteq C_8$, we need to have point of order $8$ on $E$  defined over at most quadratic extension of $\mathbb{Q}$. Since $E$ has a rational $3$-isogeny, by \cite[Table 1]{2} we see that $E$ must have a point $P_3$ of order $3$ such that $[\mathbb{Q}(P_3):\mathbb{Q}] \in \{ 1,2 \}$. Therefore a point $P_8+P_3$ of order $24$ on $E$ is defined over a field $F=\mathbb{Q}(P_3,P_8)$ such that $\Gal(F/\mathbb{Q}) \in \{ C_1, C_2, C_{2} \oplus C_2 \}$. This is impossible because of Theorem \ref{Theorem 2.1}, Theorem \ref{Theorem 2.4} and \cite[Theorem 1.4.]{22}.
\end{proof}

\begin{theorem}
 $E/\mathbb{Q}$ be an elliptic curve without $\mathrm{CM}$. Then $E(K)_{tors}$ can't contain $C_3 \oplus C_{15}$, $C_3 \oplus C_{21}$,$C_{32}$, $C_{2} \oplus C_{16}$, $C_4 \oplus C_{8}$, $C_6 \oplus C_{12}$, $C_7 \oplus C_{7}$ or $C_9 \oplus C_{9}.$
\end{theorem}

\begin{proof}

    \fbox{$C_3 \oplus C_{15}$, $C_3 \oplus C_{21}$} : Since $\mathbb{Q}(E[3]) \subseteq K$, we need to have $G_{E}(3) \in \{ \mathrm{3Cs.1.1, 3B.1.1, 3B.1.2} \}$.
If $G_{E}(3)=\mathrm{3Cs.1.1}$, Lemma \ref{Lemma 3.3} gives us that if $p \in \{ 5,7 \}$, then $E$ has a rational $p$-isogeny.  $P_p \in E(K)$ be a point of order $p$.  $\{ P_3, Q_3 \}$ be a basis for $E[3]$ such that $G_{E}(3)=\mathrm{3Cs.1.1}$ with respect to this basis. This means that $< \negthickspace P_3 \negthickspace >$ and $< \negthickspace Q_3 \negthickspace >$ are kernels of two independent rational $3$-isogenies. Now  $< \negthickspace P_p+P_3 \negthickspace >$ and $< \negthickspace Q_3 \negthickspace >$ are kernels of independent rational $3p$ and $3$ isogenies, respectively. We conclude that $E$ is isogenous to $E'/\mathbb{Q}$ with a rational $9p$-isogeny, which is impossible by Theorem \ref{Theorem 2.3}. Therefore $G_{E}(3) \in \{ \mathrm{3B.1.1, 3B.1.2} \}$.
\\Assume that $C_{3} \oplus C_{15} \subseteq E(K)$. We have that $K=\mathbb{Q}(E[3])$ and $\Gal(K/\mathbb{Q}) \cong S_3$.  $P_5$ be a point of order $5$ in $E(K)$. From \cite[Table 1]{2} we see that $[\mathbb{Q}(P_5):\mathbb{Q}] \in \{ 1,2 \}$. Denote by $F$ a unique quadratic subextension of $K$. For any possibility for $G_{E}(3)$, we have that there exists a point $P_3$ of order $3$ in $E(K)$ defined over quadratic extension of $\mathbb{Q}$. Therefore we have $\mathbb{Q}(P_5),\mathbb{Q}(P_3) \subseteq F$, so $E(F) \supseteq C_{15}$. By \cite[Theorem 2.c)]{7}, $\mathrm{LMFDB}$ label of $E$ is $\href{http://www.lmfdb.org/EllipticCurve/Q/50b1/}{\mathrm{50.b3}}$, $\href{http://www.lmfdb.org/EllipticCurve/Q/50b2/}{\mathrm{50.b4}}$, $\href{http://www.lmfdb.org/EllipticCurve/Q/50a3/}{\mathrm{50.a2}}$ or $\href{http://www.lmfdb.org/EllipticCurve/Q/450b4/}{\mathrm{450.g4}}$. None of these four curves have $C_3 \oplus C_{15}$ torsion over sextic field,  by \cite{1}.
    \\
    \\
    Assume that $C_{3} \oplus C_{21} \subseteq E(K)$. We have that $K=\mathbb{Q}(E[3])$ and $\Gal(K/\mathbb{Q}) \cong S_3$.  $P_7$ be a point of order $7$ in $E(K)$. If $[\mathbb{Q}(P_7):\mathbb{Q}] \in \{ 3,6 \}$, then by \cite[Table 1]{2} $\mathbb{Q}(P_7)$ is cyclic Galois over $\mathbb{Q}$. But $K$ is not cyclic and it does not have any Galois cubic subextensions. We conclude that $[\mathbb{Q}(P_7):\mathbb{Q}] \in \{ 1,2 \}$.  $P_3$ denote a point of order $3$ in $E(K)$ defined over at most quadratic extension of $\mathbb{Q}$. Therefore, $\mathbb{Q}(P_7),\mathbb{Q}(P_3) \subseteq F$ so $E(F) \supseteq C_{21}$, but this is impossible, by Theorem \ref{Theorem 2.4}.
\\
\\
    \fbox{$C_{32}$, $C_{2} \oplus C_{16}$, $C_4 \oplus C_8$}: By \cite[Corollary 3.5]{5}, we get that if $T$ is one of these three groups, then $[\mathbb{Q}(T):\mathbb{Q}]$ must be divisible by $4$, which is impossible since $\mathbb{Q}(T) \subseteq K$.
    \\
    \\
    \fbox{$C_6 \oplus C_{12}$}: If $G_{E}(3) \in \{ \mathrm{3B.1.1}, \mathrm{3B.1.2} \}$, we have $K=\mathbb{Q}(E[3])$, $K$ is an $S_3$ extension of $\mathbb{Q}$ and $j(E)=\frac{27(y+1)(y+9)^3}{y^3}$, for some $y \in \mathbb{Q}^{\times}$.
    \\
    If $G_{E}(2)=\mathrm{2Cn}$, then $\mathbb{Q}(E[2])$ is cubic Galois over $\mathbb{Q}$ contained in $K$, which is impossible since $K$ is an $S_3$ extension of $\mathbb{Q}$.
    \\
    If $G_{E}(2)=\Gl_{2}(\mathbb{F}_{2})$, examining the results of \cite{36} we find that if $E$ attains a point of order $4$ over sextic field and $G_{E}(2)=\Gl_{2}(\mathbb{F}_{2})$, then the image of $2$-adic representation associated to $E$ is contained in $H_{20}$, so $j(E)=\frac{(x^2-3)^3(-4x^2+32x+44)}{(x+1)^4}$. Taking the fiber product of $X_{0}(3)$ and $X_{20}$ we get a singular genus $1$ curve $C$ whose normalization is the elliptic curve $E'/\mathbb{Q}$ with $\mathrm{LMFDB}$ label $\href{http://www.lmfdb.org/EllipticCurve/Q/48a3/}{\mathrm{48.a3}}$. Inspecting the rational points on $C$ we get that there are $4$ non-singular non-cuspidal points corresponding to the $j$-invariants $109503/64$
and $-35937/4$.  Additionally, since $\mathbb{Q}(E[2]) \subseteq \mathbb{Q}(E[3])$, by \cite[Remark 1.5]{28}, we have $j(E)=2^{10}3^{3}y^3(1-4y^3)$. For $a \in \{109503/64,-35937/4 \} $ we find that $2^{10}3^{3}y^3(1-4y^3)-a=0$ has no rational solutions. Therefore, this case can't occur.
\\
Consider the case $G_{E}(2) \in \{ \mathrm{2B}, \mathrm{2Cs} \}$. There is a unique quadratic extension contained in $K=\mathbb{Q}(E[3])$, namely $\mathbb{Q}(\zeta_3)$ and we have $\mathbb{Q}(E[2]) \subseteq \mathbb{Q}(\zeta_3)$. Since every point $P_2$ of order $2$ on $E$ is defined over at most quadratic extension of $\mathbb{Q}$, by \cite[Proposition 4.8]{2} we have that a point $P_4$ of order $4$ on $E(K)$ satisfies $[\mathbb{Q}(P_4):\mathbb{Q}] \in \{ 1,2 \}$. It follows that $C_2 \oplus C_4 \subseteq E(\mathbb{Q}(E[2])) \subseteq E(\mathbb{Q}(\zeta_3))$.  Since $G_{E}(3) \in \{ \mathrm{3B.1.1}, \mathrm{3B.1.2} \}$, by \cite[Table 1]{2} we can see that there must exist a point $P_3$ in $E(K)$ such that $[\mathbb{Q}(P_3):\mathbb{Q}] \in \{ 1,2 \}$, so $P_3$ is defined over $\mathbb{Q}(E[2]) \subseteq \mathbb{Q}(\zeta_3)$. Finally, we have that $C_2 \oplus C_{12} \subseteq E(\mathbb{Q}(\zeta_3))$, but this is impossible by \cite[Theorem 1, iii)]{29}.
\\
\\
 If $G_{E}(3)=\mathrm{3Cs.1.1} \subseteq \mathrm{3Ns}$, so $j(E)=y^3$, by \cite[ Theorem 1.1]{3}. If $G_{E}(2)=\Gl_{2}(\mathbb{F}_{2})$, then again we get that $2$-adic representation associated to $E$ is contained in $H_{20}$. We have that $y^3=\frac{(x^2-3)^3(-4x^2+32x+44)}{(x+1)^4}$ induces a genus $2$ hyperelliptic curve $C$. In \texttt{Magma} \cite{35}, we compute it's  Jacobian $J(C)$ and see that it has rank $0$ over $\mathbb{Q}$. Using Chabauty method implemented in \texttt{Magma} \cite{35} we conclude that it does not have affine rational points.
\\
If $G_{E}(2)=\mathrm{2Cn}$, $j(E)=x^2+1728$ and the corresponding fiber product 
\\ $X_{2Cn} \times_{X_{0}(1)} X_{3Ns}$ is birational to $y^3=x^2+1728$, which is an elliptic curve $E'/\mathbb{Q}$ with $\mathrm{LMFDB}$ label $\href{http://www.lmfdb.org/EllipticCurve/Q/36/a/3}{\mathrm{36.a3}}$. Rational point on $E'$ corresponds to the $j$-invariant $1728$, so $E$ would have $\mathrm{CM}$, which contradicts our assumption.
\\
Consider the case $G_{E}(2) \in \{ \mathrm{2B}, \mathrm{2Cs} \}$. Using exactly the same reasoning as before, we conclude that $C_2 \oplus C_4 \subseteq E(\mathbb{Q}(E[2]))$ and $\mathbb{Q}(E[2])$ is at most quadratic over $\mathbb{Q}$. On the other hand, since $G_{E}(3)=\mathrm{3Cs.1.1}$, $\mathbb{Q}(E[3])$ is quadratic over $\mathbb{Q}$. Therefore, the composite field $L:=\mathbb{Q}(E[2])\mathbb{Q}(E[3])$ is either a $C_2 \oplus C_2$ or $C_2$ extension of $\mathbb{Q}$ and we have $C_6 \oplus C_{12} \subseteq E(L)$. But this is impossible by \cite[Theorem 1.4.]{22}.
\\
    \\
   \fbox{$C_7 \oplus C_{7}$}: Since $\mathbb{Q}(E[7]) \subseteq K$ we have $|G_{E}(7)| \le 6$, but looking at the possible mod $7$ images in \cite[Table 1]{2} we see that $|G_{E}(7)| \ge 18$, a contradiction.
   \\
   \\
    \fbox{$C_9 \oplus C_{9}$}: Since $\mathbb{Q}(E[9]) \subseteq K$ we have $|G_{E}(9)|=6$, because otherwise we would have $C_9 \oplus C_9 \in \Phi_{\mathbb{Q}}(3)$ or $C_9 \oplus C_9 \in \Phi_{\mathbb{Q}}(2)$, which isnt true by \cite{7}. Using \texttt{Magma} \cite{35}, we find all subgroups $G$ of $\Gl_{2}(\mathbb{Z}/9\mathbb{Z})$ of order $6$ such that $det(G)=(\mathbb{Z}/9\mathbb{Z})^{\times}$. All such groups $G$ are (up to conjugacy) subgroups of the group of upper triangular matrices, so $E$ has a rational $9$-isogeny. Additionally, all such groups $G$ reduce modulo $3$ to $\mathrm{3Cs.1.1}$ (up to conjugacy), which implies that $E$ has two independent rational $3$-isogenies. Therefore $E$ has independent rational $9$ and $3$-isogenies, so it's isogenous over $\mathbb{Q}$ to $E'/\mathbb{Q}$ with a rational $27$-isogeny. It follows that $E'$ has $\mathrm{CM}$ and so does $E$, which is a contradiction to our assumption.
\end{proof}

\begin{definition}\cite[Definition 3.1]{10}
We say that a finite group $G$ is of generalized $S_3$-type if it is isomorphic to a subgroup of a direct product $S_3 \times S_3 \times ... \times S_3$.
\end{definition}

\begin{theorem}\cite[Lemma 3.2, Corollary 3.4]{10} 
A finite group $G$ is of Generalized $S_3$-type if and only if
~\begin{itemize}
    \item $G$ is supersolvable
    \item Sylow subgroups of $G$ are Abelian
    \item Exponent of $G$ divides $6$.
\end{itemize}
Additionally, if $G$ is of Generalized $S_3$-type, then every subgroup and every quotient group of $G$ is also of Generalized $S_3$-type. If $G_1$ and $G_2$ is of Generalized $S_3$-type, then so is $G_1 \times G_2$.
\end{theorem}
\begin{theorem}[\cite{10}, Theorem 3.5, Theorem 3.6]
 $L$ be a number field such that $\Gal(\hat{L}/\mathbb{Q})$ is of Generalized $S_3$-type. Then $L \subseteq \hat{L} \subseteq \mathbb{Q}(3^{\infty})$.
 $L$ be a number field in $\mathbb{Q}(3^{\infty})$. Then $\hat{L} \subseteq \mathbb{Q}(3^{\infty})$ and $\Gal(\hat{L}/\mathbb{Q})$ is of Generalized $S_3$-type. 
\end{theorem}
It's easy to see that the groups $S_3$, $C_2$ and $C_3$ are of Generalized $S_3$-type and so are their direct products, $S_3 \times C_2$, $S_3 \times C_3$, $C_2 \times C_3 \cong C_6$.

\begin{theorem}
 $E/\mathbb{Q}$ be an elliptic curve without CM. Then $E(K)_{tors}$ can't contain $C_{36}$ or $C_2 \oplus C_{30}$.
\end{theorem}
\begin{proof}
 \fbox{$C_{36}$}:  $P_9,P_4$ be points of order $9$ and $4$, such that $[\mathbb{Q}(P_9+P_4):\mathbb{Q}]=6$. If $[\mathbb{Q}(P_9):\mathbb{Q}] \in \{ 1,2,3 \}$, then we have $\mathbb{Q}(P_9) \subseteq \mathbb{Q}(3^{\infty})$, since every quadratic and cubic extension is contained in $\mathbb{Q}(3^{\infty})$. If $[\mathbb{Q}(P_9):\mathbb{Q}] = 6$, we check using \texttt{Magma} \cite{35}  \space that $\Gal(\widehat{\mathbb{Q}(P_9)}/\mathbb{Q})$ is isomorphic to one of the following groups: $C_6, S_3, S_3 \times C_3, S_3 \times C_2$. All these groups are of generalized $S_3$-type, so it follows that $\mathbb{Q}(P_9) \subseteq \mathbb{Q}(3^{\infty})$. Similarly, a point $P_4$ can be defined over extensions of degree $1,2,3$ or $6$. If $[\mathbb{Q}(P_4):\mathbb{Q}] \in \{ 1,2,3 \}$, then we have $\mathbb{Q}(P_4) \subseteq \mathbb{Q}(3^{\infty})$. If $[\mathbb{Q}(P_4):\mathbb{Q}]=6 $, by the results in \cite{35} we see that if $[\mathbb{Q}(P_4):\mathbb{Q}]=6$, then $\mathbb{Q}(P_4)$ is an $S_3$ extension of $\mathbb{Q}$, hence of generalized $S_3$-type. We conclude that in any case we have $\mathbb{Q}(P_4) \subseteq \mathbb{Q}(3^{\infty})$. We can now conclude that $\mathbb{Q}(P_9+P_4) = \mathbb{Q}(P_9,P_4) \subseteq \mathbb{Q}(3^{\infty})$, which is impossible by \cite[Theorem 1.8.]{10}.
 \\
\\
\fbox{$C_2 \oplus C_{30}$}: By Lemma \ref{Lemma 3.3}, $E$ has a rational $3$ and $5$-isogenies, so it has a rational $15$-isogeny. If $E(\mathbb{Q})[2] \supseteq C_2$, then $E$ has a rational $30$-isogeny, which is impossible by Theorem \ref{Theorem 2.3}. If $G_{E}(2)=\mathrm{2Cn}$, then $j(E)=y^2+1728$ and since $E$ has $15$-isogeny we $j(E) \in \{ -5^2/2, -5^2 \cdot241^3/2^3, -29^{3} \cdot 5/2^{5}, 211^3 \cdot 5/2^{15} \}$.  $a$ be one of those $4$ values. We have that $a < 1728$, so $y^2+1728=a$ does not have a solution in real numbers, so this case is impossible.  Consider now the case when $G_{E}(2)=\Gl_{2}(\mathbb{F}_2)$. This means that $E$ attains it's full $2$-torsion over degree $6$ extension of $\mathbb{Q}$. Since $E(K)[2]=C_2 \oplus C_2$, we have $K=\mathbb{Q}(E[2]) \subseteq \mathbb{Q}(3^{\infty})$, because the Galois group of $\mathbb{Q}(E[2])$ is of generalized $S_3$-type. Therefore, $C_2 \oplus C_{30} \subseteq E(K) \subseteq E(3^{\infty})$. By \cite[Theorem 1.8., Table 1]{10}we see that $j(E) \in \{ \frac{-121945}{32}, \frac{46969655}{32768} \}$. For each of these two possiblities, using division polynomial method we calculate a polynomial $f_{30}$ whose roots are $x$-coordinates of points of exact order $30$. If $j(E)=\frac{-121945}{32}$, the smallest irreducible factors of $f_{30}$ are polynomials $f,g$ of degree $6$. Since $C_{30} \subseteq E(K)=E(\mathbb{Q}(E[2]))$, one of those polynomials needs to have a root in $K$, but since $K$ is Galois, it splits in $K$. But we check using \texttt{Magma} \cite{35} that the splitting fields of $f$ and $g$ are degree $12$-extensions of $\mathbb{Q}$, which is a contradiction. If $j(E)=\frac{46969655}{32768}$, we do the same as in the previous case. This time, the polynomial $f_{30}$ does not have irreducible factors of degree $\le 6$, so $C_{30} \not\subseteq E(K)$.
\end{proof}

\begin{theorem}
 $E/\mathbb{Q}$ be an elliptic curve without CM. If the $2$-adic representation of $E$ does not equal to $2B$, then $E(K)_{tors}$ can't contain $C_3 \oplus C_{18}$.
\end{theorem}
\begin{proof}
We will split the proof into two main cases, depending on $G_{E}(3)$.
\\
\fbox{$G_{E}(3) \in \{ \mathrm{3B.1.1, 3B.1.2} \}$}.
\\
\\
We have $K=\mathbb{Q}(E[3])$ and $K$ is $S_3$ extension of $\mathbb{Q}$.
\\If $G_{E}(2)=\mathrm{2Cs}$, this is shown to be impossible by \cite[Proposition 6.(m)]{8}.
\\If $G_{E}(2)=\mathrm{2Cn}$, then $\mathbb{Q}(E[2])$ is cubic Galois over $\mathbb{Q}$ contained in $K$, which is impossible since $K$ is an $S_3$ extension of $\mathbb{Q}$.
\\
\\
Assume first that $E$ has a rational $9$-isogeny.
\\ $G_{E}(2)=\Gl_{2}(\mathbb{F}_2)$. Since $K$ is Galois and $E$ has a point of order $2$ in $K$ and the defining cubic polynomial $f(x)$ of $E$ is irreducible and has a root in $K$, it splits in $K$. Therefore we have $K=\mathbb{Q}(E[2])=\mathbb{Q}(E[3])$. Since $E$ has a rational $9$-isogeny, by \cite[Appendix]{37}, we have that $E$ is a twist of elliptic curve \[E_t: y^2=x^3-3t(t^3-24)x+2(t^6-36t^3+216), \] where $t \in \mathbb{Q} \setminus \{3\}$. We have $j(E_t)=\frac{t^3(t^3-24)^3}{t^3-27}$ and $\Delta(E_t)=2^{12}3^6(t^3-27)$. It can be easily seen that for $t \in \{-6 , 0\}$ $E_t$ has $\mathrm{CM}$. Since $E$ is a twist of some $E_t$, we have $\Delta(E)=u^{12}\Delta(E_t)$, for some $u \in \mathbb{Q}$. The Weil pairing implies that $\mathbb{Q}(\zeta_3) \subseteq K$ and since $K$ is an $S_3$ extension of $\mathbb{Q}$, we conclude that $\Gal(K/\mathbb{Q}(\zeta_3)) \cong C_3$, which implies that the discriminant of $E$ is a square in $\mathbb{Q}(\zeta_3)$, which is equivalent to \[C: y^2=t^3-27, t \in \mathbb{Q}\setminus \{-6, 0, 3 \}, y \in \mathbb{Q}(\zeta_3) \] having a solution. Using \texttt{Magma} \cite{35} we find that such a solution does not exist. Therefore, in this case there does not exist an elliptic curve with $C_3 \oplus C_{18}$ torsion defined over $\mathbb{Q}$. 
\\ 
\\
Assume now that $E$ does not have a rational $9$-isogeny.
\\
Obviously, $E(\mathbb{Q}(3^{\infty}))$ contains a point of order $9$, since $E(\mathbb{Q}(3^{\infty})) \supseteq E(K)$. By \cite[Lemma 6.13.]{10}, we get that $j(E)=\frac{(x+3)(x^2-3x+9)(x^3+3)^3}{x^3}$.
\\
If $G_{E}(2)=\mathrm{2B}$, we have that $j(E)=\frac{256(y+1)^3}{y}$. The induced modular curve $C$ is genus $2$ hyperelliptic curve. Computation in \texttt{Magma} \cite{35} shows that points on $C$ do not correspond to elliptic curves with $C_3 \oplus C_{18}$ torsion over sextic fields. 
\\ now $G_{E}(2)=\Gl_{2}(\mathbb{F}_2)$. Since $f(x)$ is irreducible and it has a root in $K$, it splits in $K$, since $K$ is Galois. Therefore we have $\mathbb{Q}(E[2])=\mathbb{Q}(E[3])=K$. By \cite[Remark 1.5]{28}, we have $j(E)=2^{10}3^{3}y^3(1-4y^3)$, for some $y \in \mathbb{Q}$. The induced modular curve is an elliptic curve with $C_3$ torsion and rank $0$ over $\mathbb{Q}$. None of the points on this curve correspond to elliptic curves with $C_3 \oplus C_{18}$ torsion over sextic field, which is checked using \texttt{Magma} \cite{35}.
\\
\\
\fbox{$G_{E}(3)=\mathrm{3Cs.1.1}$}
Since $\mathrm{3Cs.1.1} \subseteq \mathrm{3Ns}$, we have $j(E)=y^3$.
\\
If $G_{E}(2) \in \{\mathrm{2Cn}, \mathrm{2Cs} \}$, this has already been shown to be impossible by Theorem 3.6, in $C_6 \oplus C_{12}$ case.
\\
If $G_{E}(2)=\Gl_{2}(\mathbb{F}_2)$, then we either have $K=\mathbb{Q}(E[2])$ or $K=L\mathbb{Q}(\zeta_3)$, where $L$ is degree $3$ extension of $\mathbb{Q}$ contained in $\mathbb{Q}(E[2])$. Obviously, $E(L)[2]=C_2$.
\\
If $K=L\mathbb{Q}(\zeta_3)$, then $\Gal(\hat{K}/\mathbb{Q}) \cong S_3 \times C_2$. Using \texttt{Magma} \cite{35}  we find that if $P_9$ is a point of order $9$ defined on $K$ and $G_{E}(3)=\mathrm{3Cs.1.1}$, then we can't have $[\mathbb{Q}(P_9):\mathbb{Q}]=6$, because a Galois closure of $\mathbb{Q}(P_9)$ over $\mathbb{Q}$ is one of the following groups: $C_6$, $S_3$, $S_3 \times C_3$. Since $K$ contains only two subfields, $\mathbb{Q}(\zeta_3)$ and $L$ (which we check using \texttt{Magma} \cite{35}), we either have $\mathbb{Q}(P_9) \subseteq \mathbb{Q}(\zeta_3)$, in which case $E(\mathbb{Q}(\zeta_3)) \supseteq C_3 \times C_9$ (which is impossible by Theorem 2.4) or $\mathbb{Q}(P_9)=L$. This means that $E(\mathbb{Q})=C_3$ and $E(L)=C_{18}$, but this is impossible by \cite[Theorem 2]{9}. Therefore, we need to have $K=\mathbb{Q}(E[2])$. Since $\mathbb{Q}(\zeta_3) \subseteq \mathbb{Q}(E[2])=K$, we need to have $\mathbb{Q}(\sqrt{\Delta})=\mathbb{Q}(\zeta_3)=\mathbb{Q}(\sqrt{-3})$. From this equality it follows that $\sqrt{\Delta}=\alpha+\beta \sqrt{-3}$, for some rational $\alpha, \beta$. It's easy to see that $\alpha=0$, so $\Delta=-3\beta^2$. Since $G_{E}(3)=\mathrm{3Cs.1.1}$, by \cite[Theorem 1.2.]{3} we have that $E$ is isomorphic to 
$y^2 = x^3-3(t + 1)(t + 3)(t^2 + 3)x-2(t^2-3)(t^4 + 6t^3 + 18t^2 + 18t + 9)=x^3+ax+b$, for some $t \in \mathbb{Q}$ or a quadratic twist by $-3$ of such curve. Since twisting does not change $2$-division field, we have that $\Delta=4a^3+27b^2=4(-3(t + 1)(t + 3)(t^2 + 3))^3+27(2(t^2-3)(t^4 + 6t^3 + 18t^2 + 18t + 9))^2=(t(t^2+3t+3))^3=-3\beta^2$. Plugging in $t=-3t_1$ and $\beta=\frac{\beta_{1}}{3^5}$ in $(t(t^2+3t+3))^3=-3\beta^2$ we obtain $(t_{1}(t_{1}^2-9t_{1}+27))^3=\beta_{1}^2$. Therefore, $t_{1}(t_{1}^2-9t_{1}+27)$ must be a square, so $t_{1}(t_{1}^2-9t_{1}+27)=\beta_{2}^2$, where $\beta_{2}^{6}=\beta_{1}^2$. Finally, put $t_2=t_1-3$ to obtain $t_{2}^3+27=\beta_{2}^2$, which is an elliptic curve $E'$ with $\mathrm{LMFDB}$ label $\href{http://www.lmfdb.org/EllipticCurve/Q/144/a/4}{\mathrm{144.a4}}$ and the only non trivial rational point on $E'$ is $(-3,0)$. We have that $\beta_2=0$ and so $\beta=0$, but this is impossible, because $0 \neq \Delta=-3\beta^2$.
\end{proof}
 us adress the issue that occurs in the case $G_{E}(2)=\mathrm{2B}$. Assume for convenience that $G_{E}(2)=\mathrm{3Cs.1.1}$. Using \texttt{Magma} \cite{35}, we search for possible mod $9$ images of $E$ such that $E$ has a point of order $9$ defined over sextic number field. For each possibility for $G_{E}(9)$, we find that it is contained in one of the groups with labels $9\mathrm{H}^{0}-9a$, $9\mathrm{H}^{0}-9b$ or $9\mathrm{H}^{0}-9c$ from \cite[Table 1]{34}. The modular curve induced by combining $j$-maps of one of these groups (also available in \cite[Table 1]{34}), along with $j$-map of elliptic curve $E$ with $G_{E}(2)=\mathrm{2B}$, we get a few genus $3$ and $4$ curves that are not hyperelliptic and which do not have a a nice quotient curve.

\section*{Acknowledgements}
The author would like to thank his advisor, Filip Najman, for the multitude of helpful conversations on the topic. Author would also like to thank David Zuerick-Brown, Maarten Derickx and Jackson S. Morrow.

\bibliographystyle{amsplain}

\end{document}